\documentclass[nointlimits,11pt,oneside]{amsart}
\usepackage{amssymb,cases,enumitem, verbatim}
\usepackage{color}
\usepackage{hyperref}

\hypersetup{
	colorlinks=true,
	linkcolor=blue,
	citecolor=blue,
	filecolor=green,
	urlcolor=cyan,
	bookmarks=true
}

\makeatletter
\renewcommand*{\eqref}[1]{%
	\hyperref[{#1}]{\textup{\tagform@{\ref*{#1}}}}%
}
\makeatother

\setlist[enumerate,1]{label={\textup{(\roman*)}}}

\usepackage[%
	a4paper,
	total={16cm,23cm},
	left=3cm, top=3cm,
	marginparsep=2pt]
{geometry}

\theoremstyle{plain}
\newtheorem{theorem}{Theorem}[section]
\newtheorem{lemma}[theorem]{Lemma}

\theoremstyle{definition}
\newtheorem{remark}[theorem]{Remark}

\newtheorem{definition}[theorem]{Definition}

\numberwithin{equation}{section}

\def\esssup{\operatornamewithlimits{ess\,\sup}}
\DeclareMathOperator\supp{supp}

\def\L1loc{L^1_{\text{loc}}}

\begin{document}

\title{On the smoothness of slowly varying functions}
\author{Dalimil Pe{\v s}a}

\address{Dalimil Pe{\v s}a, Department of Mathematical Analysis, Faculty of Mathematics and
	Physics, Charles University, Sokolovsk\'a~83,
	186~75 Praha~8, Czech Republic}
\email{pesa@karlin.mff.cuni.cz}
\urladdr{0000-0001-6638-0913}


\thanks{This research was supported by the Grant schemes at Charles University, reg.~No.~CZ.02.2.69/0.0/0.0/19\_073/0016935, the grants no.~P201/21-01976S and P202/23-04720S of the Czech Science Foundation, and Charles University Research program No. UNCE/SCI/023.}

\begin{abstract}
In this paper we consider the question of smoothness of slowly varying functions satisfying the modern definition that, in the last two decades, gained prevalence in the applications concerning function spaces and interpolation. We show, that every slowly varying function of this type is equivalent to a slowly varying function that has continuous classical derivatives of all orders.
\end{abstract}

\date{\today}

\maketitle

\makeatletter
   \providecommand\@dotsep{2}
\makeatother

\section{Introduction and statement of the main result}
In analysis, there are types of problems that, in their critical or limiting form, often lead to structures defined using various concrete examples of slowly varying functions. Those include, for example, real interpolation (see e.g.~\cite{BennettRudnick80}), various forms of Sobolev embeddings (see e.g.~\cite{BrezisWainger80}, \cite{CianchiPick09}, and \cite{GurkaOpic05}), and regularity of degenerate elliptic equations (see e.g.~\cite{diBlasioFeo07}). To study those problems in more generality, new structures have been introduced in recent years that are defined using slowly varying functions in abstract, with the recent papers mostly using a new definition of slowly varying functions which first appeared in \cite{GogatishviliOpic04} or \cite{GogatishviliOpic05} and which we present below as Definition~\ref{DSV}. Examples of such structures include the Lorentz--Karamata spaces (introduced in \cite{EdmundsKerman00} using an older definition of slowly varying functions and later in \cite{GogatishviliOpic04} using the modern one), various generalisations of the real interpolation spaces (as for example in \cite{GogatishviliOpic05}), and spaces of generalised smoothness (as for example in \cite{MouraNeves09}).

During the study of those new structures the question of smoothness of slowly varying function naturally appeared. While it is rather clear that the modern definition of slowly varying functions does not even guarantee continuity, it remained open whether we could approximate a given slowly varying function by a smooth function in a way that would yield equivalent structures in the respective applications. We provide a positive answer to this question in Theorem~\ref{TSmSV}, where we obtain such an approximation by a function that has continuous classical derivatives of all orders. To prove this result, we have developed a new method based on a careful analysis of the properties of slowly varying functions. In a series of lemmata, we develop tools that allow for a precise decomposition of a given slowly varying function into two parts, approximating each of those separately (using a mollification argument, where a precise treatment is needed to show that the mollified function is indeed a valid approximation), and finally recomposing the obtained functions into the desired approximation of the original function.

Let us now be more specific. In this paper, we will work with the following modern definition of slowly varying functions that have appeared in literature in recent years:

\begin{definition} \label{DSV}
	Let $b:(0,\infty) \rightarrow (0, \infty)$ be a measurable function. Then $b$ is said to be slowly varying, abbreviated s.v., if for every $\varepsilon > 0$ there exists a~non-decreasing function $b_{\varepsilon}$ and a~non-increasing function $b_{-\varepsilon}$ such that $t^{\varepsilon}b(t) \approx b_{\varepsilon}(t)$ on $(0, \infty)$ and $t^{-\varepsilon}b(t) \approx b_{-\varepsilon}(t)$ on $(0, \infty)$.
\end{definition}

Here, as well as at every other occurrence in this paper, the symbol ``$\approx$'' means that the ratio of the left- and right-hand sides is sandwiched between two finite positive constants that do not depend on the arguments of the functions in question (though they usually do depend on the functions themselves), where the equivalence is assumed to hold on the entire domains of the functions, unless specified otherwise. Furthermore, measurability is in all cases to be understood as respective to the classical Lebesgue measure.

This new definition appeared first in either \cite{GogatishviliOpic04} or \cite{GogatishviliOpic05} (it is unclear which of the papers is in fact older) and has since become standard in both the theory and applications of Lorentz--Karamata spaces (see e.g.~\cite{Bathory18}, \cite{CaetanoGogatishvili11}, \cite{EdmundsOpic08}, \cite{GogatishviliNeves10}, \cite{GogatishviliOpic04}, \cite{GurkaOpic07}, \cite{NevesOpic20}, and \cite{PesaLK}) as well as the related interpolation theory (see e.g.~\cite{Baena-MiretGogatishvili22}, \cite{GogatishviliOpic05}, and \cite{OpicGroverTBD}). It is weaker than the other two common definitions as used in, for example, \cite{BinghamGoldie87} and \cite{Zygmund35_3} (we discuss the relationships and differences in detail in Remark~\ref{RDSV}), and is therefore more suited for the applications in function spaces and interpolation, as the objects that are defined using this definition are more general than would be the case if one of the older definitions were used.

The class of functions satisfying this definition is rather rich. It includes, for example, constant positive functions and, to provide something at least slightly less trivial, the functions $t \mapsto 1+\lvert \log(t) \rvert$ and $t \mapsto 1+\log(1+\lvert \log(t) \rvert)$ occurring in the definition of Lorentz--Zygmund spaces and generalized Lorentz--Zygmund spaces, the former being introduced by Bennett and Rudnick in \cite{BennettRudnick80} while the latter were introduced by Edmunds, Gurka, and Opic in \cite{EdmundsGurka95} and later treated in great detail by Opic and Pick in \cite{OpicPick99}. As a further example of an s.v.~function, this time non-logarithmic, we present a function $b$ defined on $(0, \infty)$ by

\begin{equation*}
	b(t) = \begin{cases}
		e^{\sqrt{\log t}} & \text{for } t \in [1, \infty), \\
		e^{\sqrt{\log t^{-1}}} & \text{for } t \in (0, 1).
	\end{cases}
\end{equation*} 
Restriction of this function to the interval $(0, 1)$ was used in \cite{CianchiPick09} to characterise self-optimal spaces for Gauss-Sobolev embeddings.

No comprehensive theory of the s.v.~functions of this type has been developed, mainly because they have so far been treated only as a tool in the theories of function spaces and interpolation, so the properties developed in the various papers that use these functions do not go very deep. We provide a summary of those known results in Lemma~\ref{LSV}. In this paper we aim to take a step towards this theory by proving the following result:

\begin{theorem} \label{TSmSV}
	Let $b$ be an s.v.~function. Then there is a function $c : (0, \infty) \to (0, \infty)$ which satisfies $c \approx b$, is also s.v., and has continuous classical derivatives of all orders (i.e.~$c \in \mathcal{C}^{\infty}$).
\end{theorem}

The question of smoothness of s.v.~function arose naturally in the applications, as it is in some cases necessary (or at least convenient) to assume that the s.v.~function in question is differentiable, see e.g.~\cite{OpicGroverTBD} or \cite{PesaLK}. 

For example, as shown in \cite[Theorem~3.14]{PesaLK}, the fundamental function of the Lorentz--Karamata space $L^{\infty, \infty, b}$ is equivalent to the function $b$ itself (provided that $b$ is non-decreasing, which can be in this context assumed without loss of generality). Assuming that $b$ is smooth then allows for an explicit description of the corresponding Lorentz endpoint space, where one would otherwise only have the significantly less transparent formula based on Lebesgue--Stieltjes integration.

It is worth pointing out that a weaker version of our result has been already presented in \cite[Remark~9.2]{OpicGroverTBD}. Using the property listed below in Lemma~\ref{LSV} as \ref{LEFF}, they observed that any s.v.~function is equivalent to an absolutely continuous function. Iterating their argument, we could easily show that it is in fact equivalent to a function $c \in \mathcal{C}^k$ for any finite $k$. However, this method fails to provide an equivalent function $c \in \mathcal{C}^{\infty}$.

The paper is structured in the following way: In Section~\ref{SecPrel} we gather the already known results and put Definition~\ref{DSV} into context. Then, in Section~\ref{SecProof} we prove the main result as well as several auxiliary results, many of which are of independent interest.

\section{Slowly varying functions---a survey} \label{SecPrel}
Let us first recall some important properties of s.v.~functions that are either known or that follow from known results by straightforward generalisations.

\begin{lemma} \label{LSV}
	Let $b, b_2$ be s.v.\ functions. Then the following statements hold:
	\begin{enumerate}[label=\textup{(SV\arabic*)}]
		\item \label{SV1} The function $b^r$ is slowly varying for every $r \in \mathbb{R}$. 
		\item \label{SV2} The functions $b+b_2$, $bb_2$, $\frac{b}{b_2}$ and $t \mapsto b(\frac{1}{t})$ are	are s.v.
		\item \label{SVC} The function $b$ is bounded by positive constants on compact subsets of $(0, \infty)$.
		\item \label{CSV} Let $c > 0$, then $b(ct) \approx b(t)$ on $(0, \infty)$. More specifically, for every $\varepsilon >0$ there is a constant $C_{\varepsilon} \geq 1$ such that it holds for every $c > 0$ and every $t \in (0, \infty)$ that
		\begin{equation} \label{CSV1}
			C_{\varepsilon}^{-1} \min\{c^{-\varepsilon}, \, c^{\varepsilon} \} b(t) \leq b(ct) \leq C_{\varepsilon} \max\{c^{-\varepsilon}, \, c^{\varepsilon} \} b(t).
		\end{equation}
		\item \label{SV3} It holds for every $ \alpha \neq 0$ that
		\begin{align*}
			\lim_{t \to 0^+} t^{\alpha}b(t) = \lim_{t \to 0^+}t^{\alpha}, \\
			\lim_{t \to \infty} t^{\alpha}b(t) = \lim_{t \to \infty}t^{\alpha}.
		\end{align*}
		\item \label{SV4} Let $\alpha \neq -1$. Then
		\begin{equation*}
			\int_0^1 t^{\alpha}b(t) \: dt < \infty \iff \int_0^1 t^{\alpha} \: dt < \infty
		\end{equation*}
		and
		\begin{equation*}
			\int_1^{\infty} t^{\alpha}b(t) \: dt < \infty \iff \int_1^{\infty} t^{\alpha} \: dt < \infty.
		\end{equation*}
		Consequently,
		\begin{equation*}
			\int_0^{\infty} t^{\alpha} b(t) \: dt = \infty.
		\end{equation*}
		\item \label{LEFF} \label{LEFFH} Let $\alpha \in (0, \infty)$. Then it holds for every $t \in (0, \infty)$ that
		\begin{align}
			\int_0^{t} s^{\alpha-1} b(s) \: ds &\approx t^{\alpha} b(t), & \esssup_{s \in (0,t)} s^{\alpha} b(s) &\approx t^{\alpha} b(t), \label{LEFF1}\\
			\int_t^{\infty} s^{-\alpha-1} b(s) \: ds &\approx t^{-\alpha} b(t), & \esssup_{s \in (t,\infty)} s^{-\alpha} b(s) &\approx t^{-\alpha} b(t).  \label{LEFFH1}
		\end{align}
		\item \label{LTb} \label{LHb} Let us define, for $t \in (0, \infty)$,
		\begin{align*}
			\tilde{b}(t) &= \int_0^{t} s^{-1} b(s) \: ds, & \tilde{b}_{\infty}(t) &= \esssup_{s \in (0, t)} b(s), \\
			\hat{b}(t) &= \int_t^{\infty} s^{-1} b(s) \: ds, & \hat{b}_{\infty}(t) &= \esssup_{s \in (t, \infty)} b(s).
		\end{align*}
		Then $b \lesssim \tilde{b}$, $b \lesssim \tilde{b}_{\infty}$, $b \lesssim \hat{b}$, and $b \lesssim \hat{b}_{\infty}$ on $(0, \infty)$ and the functions $\tilde{b}$, $\tilde{b}_{\infty}$, $\hat{b}$, and  $\hat{b}_{\infty}$ are s.v.~if and only if they are finite $\lambda$-a.e.~on $(0, \infty)$. Furthermore, it holds that
		\begin{align} \label{LTb1}
			\limsup_{t \to 0^+} \frac{\tilde{b}(t)}{b(t)} = \limsup_{t \to \infty} \frac{\tilde{b}(t)}{b(t)} &= \infty, & 
			\limsup_{t \to 0^+} \frac{\hat{b}(t)}{b(t)} = \limsup_{t \to \infty} \frac{\hat{b}(t)}{b(t)} &= \infty.
		\end{align}
	\end{enumerate}
\end{lemma}

\begin{proof}
	For \ref{SV1}, \ref{SV2},  \ref{CSV}, and \ref{LEFF} see \cite[Proposition~2.2]{GogatishviliOpic05}. We note that it is not quite clear from the original formulation of \ref{CSV} in \cite[Proposition~2.2, (iii)]{GogatishviliOpic05} that the constant $C_{\varepsilon}$ depends only on $\varepsilon$, but if follows from the presented proof that it is indeed so. For \ref{SVC} see \cite[Proposition~2.1]{Bathory18} and for the first part of \ref{LTb} see \cite[Lemma~2.2]{Bathory18} and \cite[Lemma~2.1]{GurkaOpic07}. The remaining part of \ref{LTb}, i.e.~\eqref{LTb1}, has been partially proved in \cite[Lemma~2.2]{CaetanoGogatishvili11} and \cite[(2.1) and (2.2)]{NevesOpic20} and the method used in the former paper can be easily adapted to cover all the four cases mentioned in \eqref{LTb1}.
	
	The properties \ref{SV3}, and \ref{SV4} follow directly from the definition. We leave the details to the reader.
\end{proof}

\begin{remark} \label{RLC}
	The restrictions $ \alpha \neq 0$ and $ \alpha \neq -1$ in \ref{SV3} and \ref{SV4}, respectively, are necessary, as the property that $b$ is s.v.~is simply too weak to determine the behaviour of  $t^{\alpha}b(t)$ in these limiting cases. The counterexamples are easy to construct.
\end{remark}

\begin{remark}
	The property \ref{LEFF} is very important. Not only it provides an essential tool for computations, but it also gives us a finer grasp on the definition of s.v. functions. To be more specific: while the definition requires only that for every $\varepsilon > 0$ there is some pair of monotone functions that satisfies the appropriate requirements, \ref{LEFF} provides us with a specific pair, given by explicit formulas, and asserts that there is a pair of monotone functions satisfying the above mentioned requirements if and only if this specific pair satisfies them.
\end{remark}

\begin{remark} \label{RDSV}
	There exist three different definitions of slowly varying functions that can be found in the literature. The three defining conditions (for a positive $b \in M((0,\infty), \lambda)$) are:
	\begin{enumerate}
		\item \label{DSVi}
		For every $\varepsilon > 0$ the function $t^{\varepsilon}b(t)$ is non-decreasing on some neighbourhoods of zero and infinity while the function $t^{-\varepsilon}b(t)$ is non-increasing on some neighbourhoods of zero and infinity.
		\item \label{DSVii}
		For every $\varepsilon > 0$ there exists a non-decreasing function $b_{\varepsilon}$ and non-increasing function $b_{-\varepsilon}$ such that
		\begin{equation*}
			\lim_{t \to 0} \frac{t^{\varepsilon}b(t)}{b_{\varepsilon}(t)} = \lim_{t \to \infty} \frac{t^{\varepsilon}b(t)}{b_{\varepsilon}(t)} = \lim_{t \to 0} \frac{t^{-\varepsilon}b(t)}{b_{-\varepsilon}(t)} =\lim_{t \to \infty} \frac{t^{-\varepsilon}b(t)}{b_{-\varepsilon}(t)} = 1.
		\end{equation*}
		\item \label{DSViii}
		$b$ satisfies Definition~\ref{DSV}.
	\end{enumerate}
	The functions satisfying the condition \ref{DSVi} are said to belong to the Zygmund class $\mathcal{Z}$ (see \cite{Zygmund35_3} for further details) and this condition was also used in the original definition of Lorentz--Karamata spaces in \cite{EdmundsKerman00}. The condition \ref{DSVii} is the closest to the original definition of slowly varying functions as given by Karamata in \cite{Karamata30} and \cite{Karamata33}---while it is not the original definition, it is equivalent to it. Functions satisfying this condition are treated thoroughly in \cite{BinghamGoldie87}. The condition \ref{DSViii} is the one currently used in most papers concerning Lorentz--Karamata spaces and it first appeared in either \cite{GogatishviliOpic04} or \cite{GogatishviliOpic05} (the idea to use only the equivalence with monotone functions appeared originally in the paper \cite{Neves02}, but in this case the function was also required to behave the same way near zero as it does near infinity, so the definition was significantly less general).
	
	Provided that the function $b$ is assumed to be bounded on compact sets (which is very reasonable in the applications concerning functions spaces and interpolation), then clearly the validity of \ref{DSVi} implies that of \ref{DSVii} and the validity of \ref{DSVii} implies that of \ref{DSViii}. On the other hand, for any function $b$ satisfying \ref{DSVii} there is a function $b_0$ satisfying \ref{DSVi} such that
	\begin{equation*}
		\lim_{t \to 0} \frac{b(t)}{b_{0}(t)} = \lim_{t \to \infty} \frac{b(t)}{b_{0}(t)} = 1.
	\end{equation*}
	
	Some of the properties shown in Lemma~\ref{LSV} can be strengthened if we assume that $b$ satisfies \ref{DSVii}. Specifically, in \ref{CSV} and \ref{LEFF} we can replace the relation ``$\approx$" in all its occurrences with ``the ratio of the left-hand side and the right-hand side converges to $1$ at both $0$ and $\infty$" while	in \eqref{LTb1} we get that
	\begin{align*}
		\lim_{t \to 0^+} \frac{\tilde{b}(t)}{b(t)} = \lim_{t \to \infty} \frac{\tilde{b}(t)}{b(t)} &= \infty, & 
		\lim_{t \to 0^+} \frac{\hat{b}(t)}{b(t)} = \lim_{t \to \infty} \frac{\hat{b}(t)}{b(t)} &= \infty.
	\end{align*}
	For details, see \cite[Chapter~1]{BinghamGoldie87}.
\end{remark}

\begin{remark}
	Let us note, that while the modern definition of s.v.~functions (as presented in Definition~\ref{DSV}) appears to be the most appropriate for the applications in function spaces and interpolation, the more classical version presented as \ref{DSVii} in the previous remark has also proven very useful and has a wide field of applications. We shall not attempt to provide a comprehensive list, but let us mention at least the extreme value theory starting with the monograph \cite{deHaan70}, the theory of probability where the monograph \cite{Feller71} was a major contribution, and the qualitative analysis of differential equation which is described well in the monograph \cite{Maric00}. To the readers wishing to go deeper into this topic we would recommend the monographs \cite{BinghamGoldie87} and \cite{deHaanLaurens06}; the surveys \cite{JessenMikosch06}, \cite{Nikolic18}, and \cite{Radulescu07}; and the works referenced therein.
\end{remark}

\section{Proof of the main result} \label{SecProof}
In this section we prove Theorem~\ref{TSmSV}. The theorem will be proved by a series of lemmata, many of which are of independent interest.

We begin with the following lemma which shows as a special case that the class of s.v.~functions is closed with respect to the relation ``$\approx$''. This special case is certainly not new, as it can be proved immediately from the definition, but it has never been formulated explicitly (at least to our knowledge), which is quite remarkable. Our lemma provides a sort of quantitative version of this observation that will be useful later.

\begin{lemma} \label{LCCC}
	Let $b$ be an s.v.~function and let $\mathcal{C}$ be some class of positive measurable functions, defined on $(0, \infty)$ and with values in $(0, \infty)$, for which there exists some constant $C > 0$ such that we have for all $c \in \mathcal{C}$:
	\begin{equation*}
		C^{-1}c \leq b \leq Cc.
	\end{equation*}
	Then every $c \in \mathcal{C}$ is an s.v.~function and for every $\alpha \in (0, \infty)$ there is a constant $\mathfrak{C}>0$, depending only on $\alpha$, $b$, and $C$, such that it holds for every $c \in \mathcal{C}$ and every $t \in (0, \infty)$ that
	\begin{align}
		\mathfrak{C}^{-1} t^{\alpha} c(t) &\leq \int_0^{t} s^{\alpha-1} c(s) \: ds \leq \mathfrak{C} t^{\alpha} c(t), \label{LCCC1}\\
		\mathfrak{C}^{-1} t^{-\alpha} c(t) &\leq \int_t^{\infty} s^{-\alpha-1} c(s) \: ds \leq \mathfrak{C} t^{-\alpha} c(t). \label{LCCC2}
	\end{align}
\end{lemma}

\begin{proof}
	That a function $c$ is s.v. will follow once we prove the formulas \eqref{LCCC1} and \eqref{LCCC2}. Since the proof is identical in both cases, we will show only the case \eqref{LCCC1}. To this end, let us fix some $\alpha \in (0, \infty)$ and denote by $\mathfrak{B}$ some positive constant for which it holds that
	\begin{equation*}
		\mathfrak{B}^{-1} t^{\alpha} b(t) \leq \int_0^{t} s^{\alpha-1} b(s) \: ds \leq \mathfrak{B} t^{\alpha} b(t).
	\end{equation*}
	Such a constant exists, as follows from \ref{LEFF}, and it depends only on $\alpha$ and $b$. We now compute for arbitrary $c \in \mathcal{C}$
	\begin{align*}
		\int_0^{t} s^{\alpha-1} c(s) \: ds &\leq C \int_0^{t} s^{\alpha-1} b(s) \: ds \leq C \mathfrak{B} t^{\alpha} b(t) \leq C^2 \mathfrak{B} t^{\alpha} c(t), \\
		\int_0^{t} s^{\alpha-1} c(s) \: ds &\geq C^{-1} \int_0^{t} s^{\alpha-1} b(s) \: ds \geq  C^{-1} \mathfrak{B}^{-1} t^{\alpha} b(t) \geq  C^{-2} \mathfrak{B}^{-1} t^{\alpha} c(t).
	\end{align*}
	We thus obtain \eqref{LCCC1} with $\mathfrak{C} = C^2\mathfrak{B}$, i.e.~with $\mathfrak{C}$ depending only on $\alpha$, $b$, and $C$.
\end{proof}

The next lemma is mostly technical, but provides a crucial tool for later work.

\begin{lemma}\label{LShSV}
	Let $b$ be an s.v.~function satisfying
	\begin{equation} \label{LShSV1}
		\lim_{t \to 0^+} b(t) \in (0, \infty)
	\end{equation}
	and assume that the domain of $b$ is extended to the entire real line by putting
	\begin{align*}
		b(s) &=  \lim_{t \to 0^+} b(t) &\textup{ for } s \in (-\infty, 0].
	\end{align*}
	Then for every $t_0 \in \mathbb{R}$ there is a constant $C > 0$ such that for every $t_1 \in [-t_0, t_0]$ the function
	\begin{align*}
		b_{t_1}(t) &= b(t-t_1), &t \in (0, \infty),
	\end{align*}
	satisfies for every $t \in (0, \infty)$
	\begin{equation} \label{LShSV1a}
		C^{-1} \leq \frac{b_{t_1}(t)}{b(t)} \leq C.
	\end{equation}
	
	Specially, $b_{t_1}$ is an s.v.~function for any choice of $t_1 \in \mathbb{R}$.
\end{lemma}

\begin{proof}
	Consider first the case $t_1 \geq 0$. It follows from \eqref{LShSV1} that there is a constant $C_1 > 0$ and some $\delta > 0$ such that
	\begin{align*}
		C_1^{-1} &\leq \frac{b(t)}{b(0)} \leq C_1 &\textup{ for } t \in (0, \delta).
	\end{align*}
	Furthermore, it follows from \ref{SVC} that there is some constant $C_2 > 0$ such that
	\begin{align*}
		C_2^{-1} &\leq \frac{b(t)}{b(0)} \leq C_2 &\textup{ for } t \in [\delta, t_0+\delta].
	\end{align*}
	Putting this together we obtain some constant $C_0 > 0$ and some $\delta > 0$ such that
	\begin{align*}
		C_0^{-1} &\leq \frac{b(t)}{b(0)} \leq C_0 &\textup{ for } t \in (0, t_0+\delta].
	\end{align*}
	It follows that for any $t_1 \in [0, t_0]$ the function $b_{t_1}$ also satisfies
	\begin{align*}
		C_0^{-1} &\leq \frac{b_{t_1}(t)}{b(0)} \leq C_0 &\textup{ for } t \in (0, t_0+\delta],
	\end{align*}
	as $b_{t_1}((0, t_0+\delta]) \subseteq b((0, t_0+\delta]) \cup \{b(0)\}$. Hence we obtain for any such $t_1$ that
	\begin{align} \label{LShSV2}
		C_0^{-2} &\leq \frac{b_{t_1}(t)}{b(t)} \leq C_0^2 &\textup{ for } t \in (0, t_0+\delta],
	\end{align}
	where $C_0$ and $\delta$ depend on $b$ and $t_0$ but not on $t_1$.
	
	As for the remaining interval $(t_0 + \delta, \infty)$, put
	\begin{equation} \label{LShSV2.5}
		C_{\delta} = \frac{t_0 + \delta}{t_0} > 1
	\end{equation}
	and note that $C_{\delta}$ does not depend on $t_1$. Then for $t > C_{\delta} t_0 = t_0 + \delta$ we have
	\begin{equation*}
		1 \geq \frac{t - t_1}{t} \geq \frac{t - t_0}{t} > 1 - C_{\delta}^{-1} > 0,
	\end{equation*}
	i.e.~for every such $t$ there is some number $a_t \in (1-C_{\delta}, 1)$ such that $t-t_1 = a_t t$. Applying \eqref{CSV1} with $c = a_t$ and $\varepsilon = 1$, we obtain that
	\begin{equation}  \label{LShSV3}
		\frac{b_{t_1}(t)}{b(t)} = \frac{b(a_t t)}{b(t)} \in [C_1^{-1}a_t, C_1a_t^{-1}] \subseteq \left [ C_1^{-1}(1-C_{\delta}), C_1(1-C_{\delta})^{-1} \right ],
	\end{equation}
	where $C_1$ depends only on $b$.
	
	Since all the constants in \eqref{LShSV2} and \eqref{LShSV3} depend only on $b$ and $t_0$, but not on $t_1$, we may now combine those estimates to obtain $C_+ = \max\{C_0^2, C_1(1-C_{\delta})^{-1}\}$ which is independent of $t_1 \in [0, t_0]$ and for which
	\begin{equation*}
		C_+^{-1} \leq \frac{b_{t_1}(t)}{b(t)} \leq C_+,
	\end{equation*}
	for all $t_1 \in [0, t_0]$ and all $t \in (0, \infty)$.
	
	The remaining case $t_1 < 0$ is similar, so we will be briefer. We can, similarly as before, find some $C_0 > 0$ such that
	\begin{align*}
		C_0^{-1} &\leq \frac{b(t)}{b(0)} \leq C_0 &\textup{ for } t \in (0, 2t_0+\delta].
	\end{align*}
	Since for any $t_1 \in [t_0, 0)$ we have $b_{t_1}((0, t_0+\delta]) \subseteq b((0, 2t_0+\delta])$, we conclude that
	\begin{align} \label{LShSV4}
		C_0^{-2} &\leq \frac{b_{t_1}(t)}{b(t)} \leq C_0^2 &\textup{ for } t \in (0, t_0+\delta].
	\end{align}
	
	As for the remaining interval, if we define $C_{\delta}$ as in \eqref{LShSV2.5}, we obtain that it holds for $t > C_{\delta} t_0 = t_0 + \delta$ that
	\begin{equation*}
		1 \leq \frac{t - t_1}{t} \leq \frac{t + t_0}{t} < 1 + C_{\delta}^{-1}.
	\end{equation*}
	We now, by the same argument as above, arrive to the estimate
	\begin{equation}  \label{LShSV5}
		\frac{b_{t_1}(t)}{b(t)}\left [ C_1^{-1}(1+C_{\delta})^{-1}, C_1(1+C_{\delta}) \right ].
	\end{equation}
	Since neither of the constants in \eqref{LShSV4} and \eqref{LShSV5} depend on $t_1$ we may again combine them to obtain the desired conclusion for $t_1 \in [-t_0,0)$. 
	
	Together with the first step, this shows that there is a constant $C>0$, depending only on $b$ and $t_0$, such that \eqref{LShSV1a} holds for all $t_1 \in [-t_0, t_0]$ and every $t \in (0, \infty)$. The remaining part follows directly from Lemma~\ref{LCCC}.
\end{proof}

We are now suitably equipped to prove a very restricted version of our main result.

\begin{lemma} \label{LSmSV}
	Let $b$ be an s.v.~function satisfying
	\begin{equation*} \label{LSmSV1}
		\lim_{t \to 0^+} b(t) \in (0, \infty).
	\end{equation*}
	Then there is a function $c : (0, \infty) \to (0, \infty)$ satisfying $c \approx b$ that has continuous classical derivatives of all orders (i.e.~$c \in \mathcal{C}^{\infty}$) and that also satisfies
	\begin{equation*} \label{LSmSV1aa}
		\lim_{t \to 0^+} c(t) \in (0, \infty).
	\end{equation*}
	
	Furthermore, if $b$ is constant on some right neighbourhood of zero, then the function $c$ can be chosen in such a way that
	\begin{equation} \label{LSmSV1a}
		\lim_{t \to 0^+} c^{(n)}(t) = 0
	\end{equation}
	for all $n \in \mathbb{N} \setminus \{0\}$.
\end{lemma}

Note that we do not claim that $\lim_{t \to 0^+} c(t)$ is equal to $\lim_{t \to 0^+} b(t)$, only that it exists and that it is a finite non-zero number.

\begin{proof}
	Assume that the domain of the function $b$ is extended as in Lemma~\ref{LShSV} and let $\eta : \mathbb{R} \to [0, \infty)$ be a mollifying kernel, that is, let it satisfy the following three requirements:
	\begin{enumerate}
		\item $\eta \in \mathcal{C}^{\infty}$,
		\item $\supp \eta \subseteq (-1, 1)$,
		\item $\int_{-\infty}^{\infty} \eta \: d\lambda = 1$.
	\end{enumerate}
	Then we may consider the function $c : \mathbb{R} \to (0, \infty)$ defined for $t \in \mathbb{R}$ by
	\begin{equation*}
		c(t) = b*\eta(t) = \int_{-\infty}^{\infty} b(s) \eta(t-s) \: ds = \int_{-1}^1 b(t-s) \eta(s) \: ds,
	\end{equation*}
	where the last equality is due to a change of variables. Then clearly $c \in \mathcal{C}^{\infty}$ so it remains to verify that $c$ is s.v.~when restricted to $(0, \infty)$ and that $c \approx b$ on $(0, \infty)$. To this end, we fix some $\varepsilon>0$ and $t>0$ and compute by the classical Fubini's theorem (which we may use because the integrand is positive):
	\begin{equation} \label{LSmSV2}
		\int_0^{t} s^{\varepsilon-1} c(s) \: ds = \int_0^{t} \int_{-1}^1 s^{\varepsilon-1}b(s-x) \eta(x) \: dx \: ds = \int_{-1}^1 \eta(x) \int_0^{t} s^{\varepsilon-1}b(s-x) \: ds \: dx.
	\end{equation}
	Now, it follows from Lemma~\ref{LShSV} that the family of functions
	\begin{equation*}
		\mathcal{C} = \left \{ s \mapsto b(s-x); \; x \in [-1,1] \right \}
	\end{equation*}
	satisfies the requirements of Lemma~\ref{LCCC} and thus we have for all $t \in (0, \infty)$ that
	\begin{equation*}
		\int_0^{t} s^{\varepsilon-1}b(s-x) \: ds \approx t^{\varepsilon}b(t-x),
	\end{equation*}
	where the equivalence constants do not depend on $x \in [-1,1]$. By plugging this into \eqref{LSmSV2} we obtain
	\begin{equation} \label{LSmSV3}
		\int_0^{t} s^{\varepsilon-1} c(s) \: ds \approx \int_{-1}^1 \eta(x) t^{\varepsilon}b(t-x) \: dx = t^{\varepsilon}c(t).
	\end{equation}
	
	On the other hand, Lemma~\ref{LShSV} also implies that it holds for all $t \in (0, \infty)$ that $b \approx t \mapsto b(t-x)$ where again the equivalence constants do not depend on $x \in [-1,1]$. Plugging this into \eqref{LSmSV2} we obtain by \ref{LEFF} and the third property of $\eta$ that it holds for all $t \in (0, \infty)$ that
	\begin{equation} \label{LSmSV4}
		\int_0^{t} s^{\varepsilon-1} c(s) \: ds \approx \int_{-1}^1 \eta(x) \int_0^{t} s^{\varepsilon-1}b(s) \: ds \: dx \approx \int_{-1}^1 \eta(x) t^{\varepsilon}b(t) \: dx = t^{\varepsilon} b(t).
	\end{equation}
	By comparing \eqref{LSmSV3} and \eqref{LSmSV4} we may now conclude that indeed $b \approx c$ on $(0, \infty)$ and, consequently, that $c$ restricted to $(0, \infty)$ is indeed an s.v.~function. Moreover, \eqref{LSmSV1aa} follows trivially from the continuity of $c$ at $0$.
	
	As for the remaining part, if the function $b$ is constant on $(0, 1)$ then the function $c$ is constant on $(-1,0)$. Since $c \in \mathcal{C}^{\infty}$, we conclude that $c^{(n)}(0) = 0$ for all $n \in \mathbb{N} \setminus \{0\}$ and \eqref{LSmSV1a} follows. If the right neighbourhood of zero on which $b$ is constant is smaller, then we may modify our proof by using a different kernel $\eta$ with small enough support and repeat the argument to arrive at the same conclusion. We leave the details to the reader.
\end{proof}

To prove Theorem~\ref{TSmSV} in full strength, we will need the following three lemmata that provide tools for decomposing and combining s.v.~functions. Those are certainly of independent interest, as they provide much stronger tools for deriving s.v.~functions than those contained in Lemma~\ref{LSV}. We recognise that if we wanted to simply derive new s.v.~functions then we could have done it in a simpler way, but in order to prove Theorem~\ref{TSmSV} we will need those lemmata in their present form.

\begin{lemma} \label{LCSV}
	Let $b$ be an s.v.~function. Then the functions
	\begin{align*}
		b_1 &= \chi_{(0, 1]} + b \chi_{(1, \infty)}, \\
		b_2 &= b \chi_{(0, 1]} + \chi_{(1, \infty)}
	\end{align*}
	are also s.v.
\end{lemma}

\begin{proof}
	We will show only that $b_1$ is an s.v.~function, since the proof for $b_2$ is almost identical. To this end, fix some $\varepsilon >0$ and find the appropriate monotone functions $b_{\varepsilon}$ and $b_{-\varepsilon}$ that are asserted to exist by Definition~\ref{DSV}. Find some $C_1, C_2>0$ such that $C_1b_{\varepsilon}(1)>1$ and $C_2b_{-\varepsilon}(1)<1$. Then $t^{\varepsilon}b_1(t) \approx \tilde{b}_{\varepsilon}(t)$, where $\tilde{b}_{\varepsilon}$ is given by
	\begin{align*}
		\tilde{b}_{\varepsilon}(t) &=
		\begin{cases}
			t^{\varepsilon} &\textup{for } t \in (0,1], \\
			C_1b_{\varepsilon}(t) &\textup{for } t \in (1, \infty),
		\end{cases}
	\end{align*}
	i.e.~it is a non-decreasing function. Similarly, $t^{-\varepsilon}b_1(t) \approx \tilde{b}_{-\varepsilon}$, where $\tilde{b}_{-\varepsilon}$ is given by
	\begin{align*}
		\tilde{b}_{-\varepsilon}(t) &=
		\begin{cases}
			t^{-\varepsilon} &\textup{for } t \in (0,1], \\
			C_2b_{-\varepsilon}(t) &\textup{for } t \in (1, \infty),
		\end{cases}
	\end{align*}
	i.e.~it is a non-increasing function.
\end{proof}

\begin{lemma} \label{LSSV}
	Let $b$ be an s.v.~function. Then the functions $b_1$ and $b_2$, defined for $t \in (0, \infty)$ by
	\begin{align*}
		b_1(t) &= b(t+1), \\
		b_2(t) &= b\left( \frac{1}{t+1} \right)
	\end{align*}
	are also s.v.~and satisfy the estimates $b_1 \approx b$ on $(1, \infty)$ and $b_2 \approx t \mapsto b(\frac{1}{t})$ on $(1, \infty)$, (or, equivalently, $t \mapsto b_2(\frac{1}{t}) \approx b$ on $(0, 1)$).
	
	Furthermore, there are some s.v.~functions $c_1, c_2$ such that for $i = 1,2$ we have $b_i \approx c_i$ and $c_i(t) = 1$ for $t \in (0, 1]$.
\end{lemma}

\begin{proof}
	We will prove only the assertions concerning $b_1$ as the case of $b_2$ is analogous. To this end, we note that if $t >1$ then
	\begin{equation*}
		1 < \frac{t+1}{t} < 2,
	\end{equation*}
	i.e.~for every such $t$ there is some number $a_t \in (1,2)$ such that $t+1 = a_t t$. Applying \eqref{CSV1} with $c = a_t$ and $\varepsilon = 1$, we obtain that
	\begin{equation*}
		\frac{b(t+1)}{b(t)} = \frac{b(a_t t)}{b(t)} \in [C_1^{-1}a_t^{-1}, C_1a_t] \subseteq \left [ \frac{1}{2C_1}, 2C_1 \right ].
	\end{equation*}
	Since the constant $C_1$ depends only on our choice of $\varepsilon$, but not on our choice of $a_t$, we obtain that
	\begin{align} \label{LSSV1}
		b_1(t) &\approx b(t) &\textup{for } t \in (1, \infty).
	\end{align}
	
	Turning now to the interval $(0,1]$, we see that the image of this interval with respect to $b_1$ is precisely the image of the interval $(1,2]$ with respect to $b$. It thus follows from \ref{SVC} that
	\begin{align} \label{LSSV2}
		b_1(t) &\approx \chi_{(0, 1]}(t) &\textup{for } t \in (0,1].
	\end{align}
	
	Putting \eqref{LSSV1} and \eqref{LSSV2} together, we see that if we put
	\begin{align*}
		c_1(t) &=
		\begin{cases}
			1 &\textup{for } t \in (0,1], \\
			b(t) &\textup{for } t \in (1, \infty),
		\end{cases}
	\end{align*}
	then $c_1(t) \approx b_1(t)$ for $t \in (0, \infty)$. Furthermore, $c_1$ is an s.v.~function, as  follows from Lemma~\ref{LCSV}. That $b_1$ is also s.v.~now follows from Lemma~\ref{LCCC}, while the desired estimate $b_1 \approx b$ on $(1, \infty)$ follows from the formula defining $c_1$.
\end{proof}

\begin{lemma}\label{LCoSV}
	Let $b_1$ and $b_2$ be s.v.~functions satisfying
	\begin{align*}
		\lim_{t \to 0^+} b_i(t) &\in (0, \infty) &\textup{ for } i =1,2.
	\end{align*}
	Then the function $b$, defined for $t \in (0, \infty)$ by
	\begin{align*}
		b(t) &= 
		\begin{cases}
			b_2 \left( \frac{1}{t} -1 \right) &\textup{ for } t<1, \\
			1 &\textup{ for } t = 1, \\
			b_1(t-1) &\textup{ for } t > 1,
		\end{cases}
	\end{align*}
	is also s.v.~and satisfies $b \approx t \mapsto b_2(\frac{1}{t})$ on $(0, 1)$ and $b \approx b_1$ on $(1, \infty)$.
\end{lemma}

\begin{proof}
	Assume that the domains of the functions $b_i$ are extended as in Lemma~\ref{LShSV}. Then, by the same lemma, the shifted functions $t \mapsto b_i (t-1)$ are s.v.~and satisfy $t \mapsto b_i (t-1) \approx b_i$ on $(0, \infty)$. Moreover, we obtain from \ref{SV2} that the function $t \mapsto b_2(\frac{1}{t} - 1)$, and consequently also $t \mapsto b_2(\frac{1}{t} - 1)b_1(t-1)$, are both s.v.~too. Since  $t \mapsto b_1 (t-1)$ is constant on $(0, 1)$ and $t \mapsto b_2(\frac{1}{t} - 1)$ is constant on $(1, \infty)$, we obtain that $b \approx t \mapsto b_2(\frac{1}{t} - 1)b_1(t-1)$ on $(0, \infty)$ and thus we have by Lemma~\ref{LCCC} that $b$ is an s.v.~function. Finally, the remaining estimates follow directly from the above mentioned fact that $t \mapsto b_i (t-1) \approx b_i$ on $(0, \infty)$.
\end{proof}

We are now prepared to prove Theorem~\ref{TSmSV}. The idea of the proof is to use Lemma~\ref{LSSV} to decompose any given s.v.~function $b$ into two parts, apply Lemma~\ref{LSmSV} to both of those and then combine the obtained functions via Lemma~\ref{LCoSV}.

\begin{proof}[Proof of Theorem~\ref{TSmSV}.]
	We begin by using Lemma~\ref{LSSV} to find a pair of functions $b_1$ and $b_2$ that are equal to $1$ on $(0, 1)$ and that satisfy $b_1 \approx b$ on $(1, \infty)$ and $b_2 \approx t \mapsto b(\frac{1}{t})$ on $(1, \infty)$. We may now use Lemma~\ref{LSmSV} to find a pair of $\mathcal{C}^{\infty}$ functions $c_1$ and $c_2$ that satisfy for $i=1,2$ that $c_i \approx b_i$ on $(0, \infty)$ and that
	\begin{align} 
		\lim_{t \to 0^+} c_i(t) &\in (0, \infty), \label{TSmSV0.5} \\
		\lim_{t \to 0^+} c_i^{(n)}(t) &= 0, \label{TSmSV1}
	\end{align}
	for all $n \in \mathbb{N} \setminus \{0\}$. Since none of those properties break if we multiply the functions $c_i$ by finite positive constants, we may in fact require that not only do the limits in \eqref{TSmSV0.5} exist but that we actually have
	\begin{equation} \label{TSmSV2}
		\lim_{t \to 0^+} c_1(t) = \lim_{t \to 0^+} c_2(t) = 1.
	\end{equation}
	We may now apply Lemma~\ref{LCoSV} on $c_1$ and $c_2$ to obtain a single function $c$. It is simple to verify that \eqref{TSmSV1} and \eqref{TSmSV2} guarantee that $c \in \mathcal{C}^{\infty}$. Furthermore, Lemma~\ref{LCoSV} combined with our previous steps implies that $c \approx t \mapsto c_2(\frac{1}{t})  \approx t \mapsto b_2(\frac{1}{t}) \approx b$ on $(0, 1)$. Similarly, $c \approx c_1 \approx b_1 \approx b$ on $(1, \infty)$.
\end{proof}

\begin{remark} \label{RCon}
	To conclude, we would like to point out that it would be useful for applications if something could be said about the derivatives of s.v.~functions. For example, \ref{LTb} shows that given an s.v.~function $b$, one may construct a new (and essentially larger) non-decreasing s.v.~function $\tilde{b}$ by taking
	\begin{equation*}
		\tilde{b}(t) = \int_0^{t} s^{-1} b(s) \: ds.
	\end{equation*}
	It is conceivable that the inverse could also be true, i.e.~that given a smooth non-decreasing s.v.~function $b$, there would be some (essentially smaller) s.v.~function $b_0$ such that $b'(t) = t^{-1} b_0(t)$. This is indeed what happens in many of the cases that appear in applications---it~holds, for example, for the functions of the form
	\begin{equation*}
		t \mapsto \left(  \int_0^{t} s^{-1} b^q(s) \: ds \right)^{\frac{1}{q}}
	\end{equation*}
	that appear naturally as the fundamental functions of the limiting cases of Lorentz--Karamata spaces (see \cite{PesaLK} for details)---and it would also be consistent with the fact that
	\begin{equation*}
	(t \mapsto t^{\alpha} b(t))' \approx t \mapsto t^{\alpha-1} b(t)
	\end{equation*}
	for $\alpha > 0$, which follows from \ref{LEFF}. A similar thing could also possibly hold for smooth non-increasing s.v. functions, only the derivative would be $b'(t) = -t^{-1} b_0(t)$ (as corresponds to the integral defining $\hat{b}$). Regretfully, we were unable to prove any of those conjectures so the problem remains open.
	
	For an example of a situation where such a result would be useful, observe \cite[Theorem~3.14]{PesaLK} where the Lorentz endpoint space corresponding to the fundamental function of a limiting Lorentz--Karamata space is described using the derivative of the s.v.~function that defines said Lorentz--Karamata space. If the above-presented conjecture were true, this description could be made more concrete and transparent and we would obtain that the mentioned Lorentz endpoint is itself a Lorentz--Karamata space. The result would thus be much more satisfactory.
\end{remark}

\bibliographystyle{dabbrv}
\bibliography{bibliography}
\end{document}